\documentclass[a4paper,12pt,reqno]{amsart}
\usepackage{graphics}
\usepackage{extarrows}
\usepackage{cases}

\usepackage{amsthm}

\usepackage{amssymb, amsmath, amsfonts, amscd, mathtools}

\usepackage{mathrsfs, latexsym}
\usepackage[all,ps,cmtip]{xy}
\usepackage{array, longtable}
\usepackage{bm, enumitem}
\usepackage{xcolor}

\usepackage{hyperref} 
\usepackage{tikz-cd}
\usepackage{graphicx}

\newtheorem{theorem}{Theorem}[section]
\newtheorem{lemma}[theorem]{Lemma}
\newtheorem{proposition}[theorem]{Proposition}
\newtheorem{corollary}[theorem]{Corollary}

\theoremstyle{definition} 
\newtheorem{definition}[theorem]{Definition}

\theoremstyle{remark} 
\newtheorem{remark}[theorem]{Remark}

\usepackage[a4paper,top=2cm,bottom=2cm,left=3cm,right=3cm,marginparwidth=1.75cm]{geometry}

\newcommand{\bQ}{{\mathbb Q}}

\newcommand{\roundup}[1]{\lceil{#1}\rceil}

\newcommand\Vol{\text{\rm Vol}}

\newcommand\OO{{\mathcal{O}}}

\newcommand{\Supp}{\operatorname{Supp}}

\newcommand{\lct}{\operatorname{lct}}
\newcommand{\glct}{\operatorname{glct}} 
\newcommand{\mult}{\operatorname{mult}}

\numberwithin{equation}{section}

\title[Bicanonical maps of threefolds of general type]{On bicanonical maps of threefolds of general type with large volumes} 

\author{Chen Jiang}
\address{Shanghai Center for Mathematical Sciences \& School of Mathematical Sciences, Fudan University, Shanghai 200438, China}
\email{chenjiang@fudan.edu.cn}

\author{Ziqi Liu}
\address{Shanghai Center for Mathematical Sciences, Fudan University, Shanghai 200438, China}
\email{22110840007@m.fudan.edu.cn}

\begin{document}

\pagestyle{myheadings}
\markboth{\hfill Z. Liu\hfill}{\hfill \hfill}
\begin{abstract}
We prove that for any smooth projective $3$-fold of general type with canonical volume greater than $12^6$, the image of its bicanonical map has dimension at least $2$.
We also study pluricanonical maps of $3$-folds of general type with large canonical volume and fibered by $(1,2)$-surfaces or $(2,3)$-surfaces.
\end{abstract}
\maketitle

\tableofcontents

\section{Introduction}\label{section1}

For a smooth projective variety $X$ of general type, we are interested in the behavior of 
the $m$-canonical map $\Phi_{|mK_X|}$ of $X$ induced by the $m$-canonical system $|mK_X|$, which is crucial towards the birational classification of varieties of general type. 

It is predicted that for a smooth projective variety $X$ of general type of dimension $n$ with large invariants such as the canonical volume $\Vol(X)$ or the geometric genus $p_g(X)$, the $m$-canonical map $\Phi_{|mK_X|}$ 
behaves just like that of $X_0\times C$ where $X_0$ is a smooth projective variety of general type of dimension $n-1$ and $C$ is a smooth projective curve of genus $g\gg 0$. Such prediction for the birationality of $\Phi_{|mK_X|}$ (known as M\textsuperscript{c}Kernan's question) was recently confirmed by Chen and Liu \cite{CL24} (see also \cite{CJ17, CJ22}).

We are interested in more explicit results in lower dimensions. Recall that Bombieri \cite{Bombieri} showed that $\Phi_{|5K_S|}$ is birational for a smooth projective surface $S$ of general type, so it is natural to study the birationality of $\Phi_{|5K_X|}$ for a smooth projective $3$-fold $X$ of general type with large $\Vol(X)$ or $p_g(X)$. 
Chen \cite{Chen03} showed that 
$\Phi_{|5K_X|}$ is birational if $p_g(X)\geq 4$, and 
recently Chen and Ding \cite{CD25} showed that $\Phi_{|5K_X|}$ is birational if $\Vol(X)\geq 86$, significantly improving the results in \cite{Tod07, LDi12, Chen12}.

Now we focus on the bicanonical map. Recall that for a smooth projective surface $S$ of general type, $h^0(S, 2K_S)\geq 2$ by the Riemann--Roch formula (see \cite[Page~185]{Bombieri}), which implies that $\Phi_{|2K_S|}$ is a non-trivial map. It is natural to ask whether $\Phi_{|2K_X|}$ is not a pencil (namely, $d_2(X)\coloneqq\dim \overline{\Phi_{|2K_X|}(X)}\geq 2$) for a smooth projective $3$-fold $X$ of general type with large $\Vol(X)$ or $p_g(X)$.
Recently, Chen, Jiang, and Yan \cite[Theorem~1.2]{CJY24} proved that $d_2(X)\geq 2$ if $p_g(X)>201$, answering an open question of Chen and Zhang \cite{CZ08}. However, for a smooth projective $3$-fold $X$ of general type with large $\Vol(X)$, the behavior of its bicanonical map is still mysterious. To the best of our knowledge, the only known result is that $h^0(X, 2K_X)\geq 1$ if $\Vol(X)>879^3$, by Di Biagio \cite[Theorem~1.1]{LDi12} (see also \cite{Tod07}), and we even do not know whether $\Phi_{|2K_X|}$ is non-trivial for $\Vol(X)$ sufficiently large. 
 
The main goal of this paper is to show that $\Phi_{|2K_X|}$ is not a pencil for a smooth projective $3$-fold $X$ of general type with large canonical volume.
\begin{theorem}
\label{thm1.1}
Let $X$ be a smooth projective $3$-fold of general type. If $\Vol(X)>12^6$, then $\Phi_{|2K_X|}$ is not a pencil, namely,  $d_2(X)\coloneqq \dim \overline{\Phi_{|2K_X|}(X)}\geq 2$.
\end{theorem}

We briefly explain the difficulties and techniques in the proof. As in the study of $\Phi_{|5K_X|}$ in \cite{Tod07, LDi12, Chen12, CD25}, the basic idea is to create non-klt centers by constructing a boundary $\Delta\sim_{\bQ}\delta K_X$ for some $\delta<2$
and cut down the dimension of non-klt centers by the Angehrn--Siu type method.
The issue in dealing with $\Phi_{|2K_X|}$ instead of $\Phi_{|5K_X|}$ or $\Phi_{|3K_X|}$ is that, often we are not allowed to cut down the dimension of a non-klt center if the canonical volume of the non-klt center
is small (e.g., a surface with small canonical volume or a curve with geometric genus $2$), otherwise the constant $\delta$ will exceed $2$. In this case, instead of cutting down the dimension of the non-klt center, we prove certain extension theorems which allow us to lift global sections from the centers to $X$. 
Here we remark that in \cite{CL24}, the key step is to prove an extension theorem for fibrations, where many techniques such as the Kawamata--Viehweg vanishing theorem, canonical bundle formula, the Zariski--Nakayama decomposition, the Hacon--M\textsuperscript{c}Kernan extension theorem, and Birkar's lower bound of lct are involved, and the extension theorem is not effective; in this paper, we prove an effective version of extension theorem by using the Nadel vanishing theorem and for curve centers we do not require them to form a fibration.

As a by-product of our method, we can study the pluricanonical maps of $3$-folds of general type fibered by surfaces with small volumes. Here we recall that in \cite[Theorem~1.4, Theorem~1.5]{XuJ14}, Xu studied the pluricanonical maps of $3$-folds of general type not fibered by $(1,2)$-surfaces or $(2,3)$-surfaces. For example, it was shown that if a $3$-fold $X$ of general type is not fibered by $(1,2)$-surfaces and $\Vol(X)> 30^3$, then $\Phi_{|4K_X|}$ is birational. We
prove the following theorem concerning the pluricanonical maps of $3$-folds of general type fibered by $(1,2)$-surfaces or $(2,3)$-surfaces, as a complement of the results of Xu:
\begin{theorem}\label{thm1.3}
Let $X$ be a smooth $3$-fold of general type. Suppose that $X$ admits a fibration $X\to C$ to a curve with a general fiber denoted by $F$.
\begin{enumerate}
 \item If $F$ is a $(1,2)$-surface and $\Vol(X)> 60$, then $\Phi_{|mK_X|}$ is generically finite of degree $2$ for $2\leq m \leq 4$ and is birational for $m\geq 5$.
 \item If $F$ is a $(2,3)$-surface and $\Vol(X)>456$, then $\Phi_{|mK_X|}$ is generically finite of degree $2$ for $2\leq m \leq 3$ and is birational for $m\geq 4$.
\end{enumerate}
\end{theorem}

\section{Preliminaries}
Throughout this paper, we work over the complex number field $\mathbb C$.
We will freely use the basic notation in \cite{KMM, K-M}.

\subsection{Conventions}
A {\it fibration} is a projective surjective morphism between normal varieties with connected fibers.

A $\mathbb{Q}$-divisor is said to be {\it $\mathbb{Q}$-effective} if it is $\mathbb{Q}$-linearly equivalent to an effective $\mathbb{Q}$-divisor.

For a $\mathbb{Q}$-Cartier Weil divisor $D$ on a normal projective variety $X$, $\Phi_{|D|}$ denotes the rational map induced by the linear system $|D|$ on $X$. 

For a smooth projective variety $X$, its {\it geometric genus} is defined by $p_g(X):=h^0(X, K_X)$ and its {\it canonical volume} is defined by $\Vol(X):=\Vol(X, K_X)$ (see Definition~\ref{def volume}). In general, for a projective variety $X$, its {\it geometric genus} is defined by $p_g(X):=p_g(X')$ and its {\it canonical volume} is defined by $\Vol(X):=\Vol(X')$ where $X'\to X$ is a resolution. $X$ is said to be {\it of general type} if $\Vol(X)>0$.

 A variety $X$ is said to be {\it minimal} if $X$ is a $\mathbb{Q}$-factorial terminal projective variety with $K_X$ nef.

An $(a,b)$-surface is a smooth projective surface $S$ of general type with $\Vol(S)=a$ and $p_g(S)=b$.

\subsection{Volumes}We recall the definition of volume of a divisor.
\begin{definition}\label{def volume} 
Let $X$ be an $n$-dimensional projective variety and $D$ be a Cartier divisor on $X$. The {\it volume} of $D$ is the real number
$$
{\Vol}(X, D)=\limsup_{m\rightarrow \infty}\frac{h^0(X, mD )}{m^n/n!}.
$$
Note that the limsup is actually a limit. Moreover by the homogenous property of volumes, we can extend the definition to $\bQ$-Cartier $\bQ$-divisors. 
It is known that $\Vol(X, D)=D^n$ if $D$ is a nef $\mathbb{Q}$-Cartier $\mathbb{Q}$-divisor.
For more background on volumes, see \cite[2.2.C]{Positivity1} and \cite[11.4.A]{Positivity2}.
\end{definition}

The following lemma is well-known to experts.
\begin{lemma}[{See \cite[Lemma~2.5]{Jiang-AJM}}]\label{lem:volume}
 Let $X$ be a normal projective variety and let $X\to C$ be a fibration to a curve with a general fiber denoted by $F$. Let $D$ be a $\mathbb{Q}$-Cartier $\mathbb{Q}$-divisor on $X$. Then for any rational number $t>0$,
\[
\Vol(X, D-tF)\geq \Vol(X, D)-t\dim X\cdot \Vol(F, D|_F).
\]
\end{lemma}
\begin{proof}
Take a sufficiently divisible positive integer $m$ such that $mt$ is an integer and $mD$ is Cartier. Note that $F|_F\sim 0.$
By considering the exact sequences 
\[
0\to H^0(X, mD-kF)\to H^0(X, mD-(k-1)F)\to H^0(F, mD|_F)
\]
for $1\leq k\leq mt$, we have 
\begin{align*}
 h^0(X, mD-mtF) \geq h^0(X, mD)-mt\cdot h^0(F, mD|_F).
\end{align*}
Dividing by $\frac{m^{\dim X}}{
(\dim X)!}$ and
taking limit, we get the desired inequality.
\end{proof}
\subsection{Singularities}
We recall basic knowledge of singularities.
\begin{definition} A {\it pair} $(X, \Delta)$ consists of a normal variety $X$ and an effective $\mathbb{Q}$-divisor $\Delta$ such that $K_X+\Delta$ is $\mathbb{Q}$-Cartier. 

A prime divisor $E$ \emph{over} $X$ is a prime divisor on some model $Y$ where $f\colon Y\to X$ is a proper birational morphism. The image $f(E)$ is called the \emph{center} of $E$ on $X$ and is denoted by $c_X(E)$. 
Write
\[
 K_Y=f^*(K_X+\Delta)+\sum_Ea(E,X,\Delta)E,
\]
where $E$ runs over prime divisors on $Y$, then $a(E, X, \Delta)\in \mathbb Q$ is called the \emph{discrepancy} of $E$.
Then $(X,\Delta)$ is said to be 
\begin{enumerate}
\item \emph{terminal} if $a(E, X, \Delta)> 0$ for all prime exceptional divisor $E$ over $X$;
 \item \emph{klt} (short for {\it Kawamata log terminal}) 
 if $a(E, X, \Delta)> -1$ for all prime divisor $E$ over $X$;
 \item \emph{plt} (short for {\it purely log terminal}) if $a(E, X, \Delta)> -1$ for all prime exceptional divisor $E$ over $X$;
 \item \emph{lc} (short for {\it log canonical}) if $a(E, X, \Delta) \geq -1$ for all prime divisor $E$ over $X$.
\end{enumerate}
Usually we say that $X$ is terminal (resp. klt, plt, lc) if $(X, 0)$ is terminal (resp. klt, plt, lc). 
\end{definition}

\begin{definition}
Let $(X,\Delta)$ be a pair. 
A closed subvariety $Z\subset X$ is called a {\em non-klt center} of $(X, \Delta)$ if there exists a prime divisor $E$ over $X$ 
such that $a(E,X,\Delta)\leq -1$ and $c_X(E)=Z$; the divisorial valuation induced by such $E$ is called a {\it non-klt place} over $Z$.

A non-klt center is called {\it minimal} (resp., {\it maximal}, {\it isolated}) if it does not contain (resp., is not contained in, does not intersect with) any other non-klt center.
\end{definition}

We recall well-known basic properties of non-klt centers on lc pairs. 
\begin{proposition}[{See \cite{Amb06}}]\label{prop non-klt center}
 Let $(X,\Delta)$ be an lc pair.
 \begin{enumerate}
 \item Any irreducible component of the intersection of $2$ non-klt centers of $(X, \Delta)$ is a non-klt center of $(X, \Delta)$.

\item Any minimal non-klt center of $(X, \Delta)$ is normal. 

\item If $Z$ is a non-klt center of $(X, \Delta)$ and it is both maximal and minimal, then it is isolated.

\item If $Z$ is a non-klt center of $(X, \Delta)$ with a unique non-klt place over $Z$, then $Z$ is a maximal non-klt center.
 
 \end{enumerate}
\end{proposition}
\begin{proof}
 (1) This is \cite[Theorem~1.1(2)]{Amb06}. 
 
 (2) This is \cite[Theorem~1.1(4)]{Amb06}. 

(3) By the maximality, $Z$ is not contained in any other non-klt center.
So if $Z$ intersects another non-klt center, then it properly contains another non-klt center by (1), which contradicts to the minimality of $Z$.

(4) 
Take $\pi:W\to X$ to be a log resolution of $(X, \Delta)$ and write $K_W+\Delta_W=\pi^*(K_X+\Delta)$.
Suppose that $E$ is the unique prime divisor on $W$ such that $\mult_E\Delta_W=1$ and $c_X(E)=Z$. 
We may write $\Delta_W=\Delta_W^+-\Delta_W^-$
 where $\Delta_W^+$ and $\Delta_W^- $ are effective $\mathbb{Q}$-divisors without common irreducible component.
After shrinking $X$ around the generic point of $Z$, we may assume that the center of any irreducible component of $\lfloor\Delta_W^+\rfloor$ contains $Z$.
Suppose that $\lfloor\Delta_W^+\rfloor\neq E$, then by the connectedness lemma \cite[Theorem~5.48]{K-M}, there exists a prime divisor $E'\neq E$ on $W$
such that $\mult_{E'}\Delta_W=1$ and $E\cap E'$ dominates $Z$, then the blowing up along $E\cap E'$ induces another non-klt place of $(X,\Delta)$ over $Z$, a contradiction to the uniqueness of $E$. 
Hence $E$ is the unique non-klt place of $(X,\Delta)$ after shrinking, which proves the maximality of $Z$. 
\end{proof}

\subsection{Multiplier ideal sheaves}
We recall the definition of multiplier ideal sheaves and their relation with non-klt centers.
\begin{definition}[{See \cite[Definition~9.3.56]{Positivity2}}] Let $(X,\Delta)$ be a pair. Let $\pi : W \to X$ be a log resolution of $(X, \Delta)$. Then the {\it multiplier ideal sheaf} of $(X,\Delta)$ is defined as 
$$
\mathcal{J}(X,\Delta):=\pi_*\OO_{W}(K_W-\lfloor\pi^*(K_X+\Delta)\rfloor)\subset\OO_X.
$$
By definition, the cosupport of $\mathcal{J}(X, \Delta)$ is exactly the union of non-klt centers of $(X, \Delta)$. 
\end{definition}
 We recall the Nadel vanishing theorem (see \cite[Theorem~9.4.17]{Positivity2}). 
\begin{theorem}[Nadel vanishing theorem]\label{nadel} Let $(X, \Delta)$ be a projective pair and let $D$ be a Cartier divisor on $X$ such that $D-K_X-\Delta$ is nef and big.
Then 
$$
H^i(X,\OO_X(D)\otimes\mathcal{J}(X, \Delta))=0
$$
for any integer $i>0$.
\end{theorem}

Note that in the Nadel vanishing theorem, $D$ is assumed to be Cartier. 
When $D$ is not necessarily Cartier, we have the following generalization, which is well-known to experts. 
\begin{lemma} \label{lem nadel2} Let $(X, \Delta)$ be a projective pair and let $D$ be a $\mathbb{Q}$-Cartier $\mathbb{Q}$-divisor on $X$ such that $D-K_X-\Delta$ is nef and big. Let $\pi : W \to X$ be a log resolution of $(X, \Delta+D)$. 
Then 
\begin{enumerate}
 \item For any integer $i>0$, $$
H^i(X,\pi_*\OO_{W}(K_W+\lceil\pi^*(D-K_X-\Delta)\rceil))=0.
$$

\item Suppose that $U\subset X$ is an open subset such that $D|_U$ is integral and Cartier, then
\[
\pi_*\OO_{W}(K_W+\lceil\pi^*(D-K_X-\Delta)\rceil)|_U\simeq \OO_U(D)\otimes\mathcal{J}(X, \Delta)|_U.
\]
\end{enumerate}
\end{lemma}
\begin{proof}
(1) By the (relative) Kawamata--Viehweg vanishing theorem \cite[Theorem~1.2.5]{KMM},
 for any integer $i>0$,
 \begin{align*}
 H^i(W, K_W+\lceil\pi^*(D-K_X-\Delta)\rceil)=0,\\
 R^i\pi_*\OO_{W}(K_W+\lceil\pi^*(D-K_X-\Delta)\rceil)=0.
 \end{align*}
So we get the conclusion by a standard spectral sequence argument. 
 
(2) The conclusion follows directly by the projection formula as $D|_U$ is Cartier. To do the computation precisely, we may assume that $X=U$ after shrinking $X$, then as $\pi^*D$ is Cartier, 
\begin{align*}
  {}&  \pi_*\OO_{W}(K_W+\lceil\pi^*(D-K_X-\Delta)\rceil)\\
    ={}&\pi_*\OO_{W}(K_W+\pi^*D+\lceil-\pi^*(K_X+\Delta)\rceil)\\
    \simeq{}&\OO_X(D)\otimes\pi_*\OO_{W}(K_W+\lceil-\pi^*(K_X+\Delta)\rceil)\\ 
 ={}&    \OO_X(D)\otimes\mathcal{J}(X, \Delta).
\end{align*}
The proof is complete.
\end{proof}

The following lemma allows us to compute the multiplier ideal sheaf explicitly under good conditions, which follows from the proof of the connectedness lemma \cite[Theorem~5.48]{K-M}.

\begin{lemma}\label{lem J=I}
 Let $(X,\Delta)$ be an lc pair. Suppose that $(X,\Delta)$ has a unique non-klt center $Z$ with a unique non-klt place over $Z$. Then 
 $
 \mathcal{J}(X, \Delta) =\mathcal{I}_Z,
 $ 
 where $\mathcal{I}_Z$ is the ideal sheaf of $Z$.
\end{lemma}
\begin{proof}
Take a log resolution $\pi:W\to X$ of $(X, \Delta)$ and let $E$ be the prime divisor corresponding to the unique non-klt place over $Z$.
By assumption, we may write
\[
K_W+E+G=\pi^*(K_X+\Delta)
\]
where all coefficients of $G$ are less than $1$. 
Then 
\[
K_W-\lfloor\pi^*(K_X+\Delta)\rfloor=-E-\lfloor G\rfloor,
\]
where $-\lfloor G\rfloor\geq 0$ is an effective $\pi$-exceptional divisor. 
By the relative Kawamata--Viehweg vanishing theorem \cite[Theorem~1.2.5]{KMM},
\[R^1\pi_*\OO_{W}(K_W-\lfloor\pi^*(K_X+\Delta)\rfloor)=0.\]
Hence by the short exact sequence 
 \[
 0\to \OO_{W}(K_W-\lfloor\pi^*(K_X+\Delta)\rfloor)\to \OO_W(-\lfloor G\rfloor)\to \OO_E(-\lfloor G\rfloor)\to 0,
 \]
 we have an exact sequence 
\begin{align}\label{eq JGGR}
 0\to \mathcal{J}(X, \Delta)\to \pi_*\OO_W(-\lfloor G\rfloor)\to \pi_*\mathcal{O}_E(-\lfloor G\rfloor)\to 0. 
 \end{align}

As $-\lfloor G\rfloor$ is effective and $\pi$-exceptional, we have 
$\pi_*\OO_W(-\lfloor G\rfloor)=\OO_X$. 
Hence by \eqref{eq JGGR}, we have a surjective map $ \OO_X\to \pi_*\mathcal{O}_E(-\lfloor G\rfloor)$ which is naturally factored as 
\[\OO_X\to \mathcal{O}_Z \hookrightarrow\pi_*\mathcal{O}_E \hookrightarrow\pi_*\mathcal{O}_E(-\lfloor G\rfloor).\] 
So $\mathcal{O}_Z=\pi_*\mathcal{O}_E=\pi_*\mathcal{O}_E(-\lfloor G\rfloor)$. Hence \eqref{eq JGGR} implies that \[\mathcal{J}(X,\Delta)=\textrm{Ker}(\OO_X\to \OO_Z)=\mathcal{I}_Z.\] 
\end{proof}

\subsection{Tankeev's principle}
In the study of geometry of linear systems, the so-called {\it Tankeev's principle} says that for a linear system $|D|$ on a variety $X$ with a fibration structure, the property (e.g., birationality) of $\Phi_{|D|}$ restricting on a general fiber could be lifted to $X$ under certain natural conditions, see \cite{Tank}.

We prove variant versions of Tankeev's principle on the behavior of linear systems.
\begin{proposition}\label{prop tankeev}
Let $X$ be a normal projective variety with a fibration $f: X\to C$ to a curve. 
Let $D$ be a $\mathbb{Q}$-Cartier Weil divisor such that the restriction map 
\begin{equation}\label{equationextF1F2}
 H^0(X, D)\to H^0(F_1, D|_{F_1}) \oplus H^0(F_2, D|_{F_2}) 
\end{equation}
is surjective for distinct general fibers $F_1,F_2$ of $f$.
If $|D|_F|$ induces a generically finite map of degree $d$ on a general fiber $F$ of $f$, then $|D|$ induces a generically finite map of degree $d$ on $X$.
\end{proposition}
\begin{proof}
 After replacing $X$ by a higher resolution 
 and replacing $D$ by its movable part, we may assume that $X$ is smooth and $D$ is free.
In particular, $D|_F$ is free on a general fiber $F$ of $f$. 
Then the surjectivity of \eqref{equationextF1F2} implies that $\Phi_{|D|}$ separates any $2$ closed points on $2$ distinct general fibers of $f$, hence $\Phi_{|D|}$ induces a generically finite map of degree $d$ on the complement of finitely many fibers of $f$.
\end{proof}

\begin{lemma}\label{lem separete Z1Z2}
Let $X$ be a normal projective variety and let $D$ be an effective Cartier divisor on $X$.
Let $C_1$ and $C_2$ be closed subvarieties on $X$.
Suppose that the natural restriction map 
\[
 H^0(X,D) \to H^0(C_1, D|_{C_1} ) \oplus H^0(C_2, D|_{C_2})
\]
is surjective.
Then \[
\dim\overline{\Phi_{{|D|}}(X)}\geq \min\{\dim\overline{\Phi_{{|D|_{C_1}|}}(C_1)}, \dim\overline{\Phi_{{|D|_{C_2}|}}(C_2)}\}+1.\]
In particular,
if $h^0(C_i, D|_{C_i})\geq 2$ for $i=1,2$, then $\dim\overline{\Phi_{{|D|}}(X)}\geq 2$. 
\end{lemma}

\begin{proof}
 By assumption, $\dim\overline{\Phi_{{|D|}}(C_i)}=\dim\overline{\Phi_{{|D|_{C_i}|}}(C_i)}$ for $i=1,2$ and $\overline{\Phi_{{|D|}}(C_1)}\neq \overline{\Phi_{{|D|}}(C_2)}$, hence the conclusion follows. 
\end{proof}
 
We need the following lemma to check the surjectivity of restriction maps, which is an easy consequence of the five lemma. 
\begin{lemma}\label{lem rest 5 lemma}
Let $X$ be a normal projective variety and let $D$ be a Cartier divisor on $X$.
Let $C_1$ and $C_2$ be closed subvarieties on $X$.
Suppose that the natural restriction maps \[H^0(X, D) \to H^0(C_1, D|_{C_1}) \text{ and } H^0(X, \mathcal{O}_X(D)\otimes \mathcal{I}_{C_1}) \to H^0(C_2, D|_{C_2})\]
are surjective. 
Then the natural restriction map
\[
 H^0(X, D) \to H^0(C_1, D|_{C_1} ) \oplus H^0(C_2, D|_{C_2})
\]
is surjective.
\end{lemma}
\begin{proof} 
We just apply the five lemma to the commutative diagram:
{ \[\begin{tikzcd}[column sep=small]
	0 & H^0(X, \OO_X(D)\otimes \mathcal{I}_{C_1}) & H^0(X, D) & H^0(C_1, D|_{C_1}) & 0 \\
	0 & H^0(C_2, D|_{C_2}) & H^0(C_1,D|_{C_1}) \oplus H^0(C_2, D|_{C_2}) & H^0(C_1,D|_{C_1}) & 0.
	\arrow[from=1-1, to=1-2]
	\arrow[from=1-2, to=1-3]
	\arrow[from=1-2, to=2-2]
	\arrow[from=1-3, to=1-4]
	\arrow[from=1-3, to=2-3]
	\arrow[from=1-4, to=1-5]
	\arrow[from=1-4, to=2-4]
	\arrow[from=2-1, to=2-2]
	\arrow[from=2-2, to=2-3]
	\arrow[from=2-3, to=2-4]
	\arrow[from=2-4, to=2-5]
\end{tikzcd}\]}
The proof is complete.
\end{proof}

\section{Glct for minimal surfaces of general type}
We recall the definition of global log canonical threshold (glct) and a result of Koll\'ar in \cite{Noether, CJY24} which gives explicit lower bounds for glct of minimal surfaces of general type.

\begin{definition}
Let $X$ be a normal variety with lc singularities and let $\Delta\geq 0$ be an effective $\bQ$-Cartier $\bQ$-divisor. The {\it log canonical threshold} ({\it lct}, for short) of $\Delta$ with respect to $X$ is defined by
$$\lct(X; \Delta) = \sup\{t\geq 0 \mid (X, t\Delta) \text{ is lc}\}.$$ 
\end{definition}
\begin{definition}[{\cite[Appendix~A]{Noether}}]\label{def glct}
Let $Y$ be a normal projective variety with klt singularities such that $K_Y$ is nef and big. The {\it global log canonical threshold} ({\it glct}, for short) of $Y$ is defined as: 
\begin{align*}
\glct(Y){}&:=\inf\{\lct(Y; \Delta)\mid 0\leq \Delta\sim_\bQ K_Y\}\\
{}&=\sup\{t\geq 0\mid (Y, t\Delta) \text{ is lc for all }0\leq \Delta\sim_\bQ K_Y\}.
\end{align*}
\end{definition}
\begin{remark}\label{rem glct}
 In Definition~\ref{def glct}, we may replace $\mathbb{Q}$-linear equivalence by numerical equivalence, that is, \begin{align}\glct(Y){}=\inf\{\lct(Y; \Delta)\mid 0\leq \Delta\equiv K_Y \text{ and } \Delta \text{ is a }\mathbb{Q}\text{-divisor}\}.\label{eq glct equiv}\end{align}
This is well-known to experts, but we give a proof for the reader's convenience. Denote by $t$ the right-hand side of \eqref{eq glct equiv}, then it is clear that $\glct(Y){}\geq t$. 
So it suffices to show that $\glct(Y){}\leq t$, which is equivalent to showing that $\glct(Y)\leq \lct(Y; \Delta)$ for
 any $\mathbb{Q}$-divisor $\Delta$ with $0\leq \Delta\equiv K_Y$. As $K_Y$ is big, for any rational number $ \epsilon>0$, we can find an effective $\mathbb{Q}$-divisor $D_\epsilon\sim_{\mathbb{Q}} \epsilon K_Y+K_Y-\Delta$, that is, $ \frac{1}{1+\epsilon}(D_\epsilon+\Delta)\sim_{\mathbb{Q}} K_Y$. Hence 
\begin{align*}
\glct(Y)\leq {}&\lct \left(Y; \frac{1}{1+\epsilon}(\Delta+D_\epsilon) \right)\\={}&(1+\epsilon)\lct(Y; \Delta+D_\epsilon)\leq (1+\epsilon)\lct(Y; \Delta).
\end{align*}
As $\epsilon$ is arbitrary, we conclude that $\glct(Y)\leq \lct(Y; \Delta)$.
\end{remark}

Now we recall Koll\'ar's lower bound for glct.
\begin{proposition}[{Koll\'ar's lower bound}] \label{prop glct lower bound}
Let $S$ be a minimal surface of general type. Then $\glct(S)\geq \frac{c_0}{K_S^2}$ where $c_0$ is listed below according to $P_2:=h^0(S, 2K_S)$:
$$
\begin{array}{rccccccccc}
P_2= & 2 & 3 & 4 & 5 & 6 & 7{-}9 & 10{-}13 & \geq 14\\[1ex]
c_0= & \tfrac1{11} & \tfrac1{13} & \tfrac1{16} &\tfrac1{17} & \tfrac1{19} & \tfrac1{P_2+14} & \tfrac1{P_2+15} & \tfrac1{2P_2+1}
\end{array}
$$
\end{proposition}
\begin{proof}
Denote by $S_c$ the canonical model of $S$, then 
$h^0(S_c, 2K_{S_c})=h^0(S, 2K_{S})=P_2$.
Applying \cite[Proposition~3.1]{CJY24} to $H=K_{S_c}$, we have
\[
\text{mcd}\left(\lct\left(S_c;\frac{1}{K_{S_c}^2}\Delta \right)\right)\leq P_2\]
for any effective $\mathbb{Q}$-divisor $\Delta\equiv K_{S_c}$, where $\text{mcd}$ is defined in \cite[(A.1.1)]{Noether}.
So by \cite[Proposition~A.4]{Noether}, we get
\[
\lct\left(S_c;\frac{1}{K_{S_c}^2}\Delta \right)\geq c_0,
\]
or equivalently, 
\[\lct(S_c; \Delta )=\frac{1}{K_{S_c}^2}\lct\left(S_c;\frac{1}{K_{S_c}^2}\Delta \right) \geq \frac{c_0}{K_S^2}\]
for any effective $\mathbb{Q}$-divisor $\Delta\equiv K_{S_c}$.
Hence $ 
\glct(S)=\glct(S_c)\geq \frac{c_0}{K_S^2}.$
\end{proof}

As applications, we have the following $2$ lemmas. 
\begin{lemma}\label{429}
Let $S$ be a minimal surface of general type with $K_S^2\leq 11$. Then $\glct(S)\geq \frac{1}{429}$.
\end{lemma}
\begin{proof}
By the Noether inequality, we have $p_g(S)\leq \lfloor\frac{1}{2}K_S^2+2\rfloor\leq 7$. 
By the Kawamata--Viehweg vanishing theorem and the Riemann--Roch formula, we have 
\[
 h^0(S, 2K_{S})=\chi (S, 2K_{S})=K_{S}^2+ 1-q(S)+p_g(S)\leq 19.
\]
Then by Proposition~\ref{prop glct lower bound}, $\glct(S)\geq \frac{1}{11\times 39}=\frac{1}{429}$.
\end{proof} 
\begin{lemma}\label{lemma 2,3}
 Let $S$ be a minimal $(2,3)$-surface. Then $\glct (S)\geq \frac{1}{38}$.
\end{lemma}
\begin{proof}
By the Kawamata--Viehweg vanishing theorem and the Riemann--Roch formula, we have 
\[
 h^0(S, 2K_{S})=\chi (S, 2K_{S})=K_{S}^2+ 1-q(S)+p_g(S)= 6.
\]
Then by Proposition~\ref{prop glct lower bound}, $\glct(S)\geq \frac{1}{2\times 19}=\frac{1}{38}$.
\end{proof} 

\section{An extension theorem for fibrations over curves and applications}
In this section, we establish an extension theorem for $3$-folds admitting a fibration over a curve (Theorem~\ref{thm:ext for surface}) and apply it to study fibrations of $(1,2)$-surfaces and $(2,3)$-surfaces (Theorem~\ref{thm1.3}) and bicanonical maps (Corollary~\ref{cor X/C}).
\subsection{An extension theorem for fibrations over curves}
\label{subsection4.1}

Recall the following result from \cite{Noether, Noether_Add} on minimality of relative minimal models in terms of glct. 

\begin{proposition}\label{prop CCJ fibration}
Let $X$ be a minimal $3$-fold of general type.
Assume that there exists a resolution $\pi:W\to X$ and a fibration $f:W\to C$ to a curve. Denote by $F$ a general fiber of $f$ and $F_0$ the minimal model of $F$. 
Assume that 
$\pi^*K_X-bF$ is $\mathbb{Q}$-effective
for some rational number $b\geq \frac{1}{\glct(F_0)}$. 
Denote by $W_0$ a relative minimal model of $W$ over $C$.
Then $W_0$ is minimal, namely, $K_{W_0}$ is nef. 
\end{proposition}

\begin{proof}
By \cite[Proposition~2]{Noether_Add}, the assumption $b\geq \frac{1}{\glct(F_0)}$ implies that there does not exist any $\pi$-exceptional prime divisor $E_0$ on $W$ with $(\pi^*(K_X)|_F\cdot E_0|_F)>0$.
Then the proposition follows from \cite[Lemma~3.2]{Noether}.
\end{proof}

By applying the Nadel vanishing theorem, we establish an extension theorem for $3$-folds admitting a fibration to a curve in terms of glct of general fibers. 

\begin{theorem}\label{thm:ext for surface}

Let $W$ be a smooth projective $3$-fold of general type and let $f:W\to C$ be a fibration to a curve. Fix a positive integer $k$. 
Suppose that
\[\Vol(W)>\frac{3k\Vol(F)}{\min\{\glct(F_0), 1\}}\]
holds for any general fiber $F$ of $f$ where $F_0$ is the minimal model of $F$. Then the natural restriction map 
\[ H^0(W, mK_W)\to \bigoplus_{i=1}^kH^0(F_{i}, mK_{F_{i}}) \]
is surjective for any integer $m\geq 2$, where $F_{i}(1\leq i\leq k)$ are distinct general fibers of $f$.
\end{theorem} 

\begin{proof}
Take $X$ to be a minimal model of $W$ and take $W_0\to C$ to be a relative minimal model of $W$ over $C$. 
After replacing $W$ by a higher resolution, 
we may assume that there are natural morphisms $\pi:W\to X$ and $\mu: W\to W_0$. As
$K_W-\pi^*K_X$ is an effective $\pi$-exceptional $\mathbb{Q}$-divisor, we have
\begin{align*}
 \Vol(W)=&\Vol(W, K_W)=\Vol(W, \pi^*K_X),\\
 \Vol(F)=&\Vol(F, K_F)\geq \Vol(F, \pi^*K_X|_F).
\end{align*} 
Take a rational number $w<{\min\{\glct(F_0), 1\}}$ such that \[\Vol(W)>\frac{3k\Vol(F)}{w}>\frac{3k\Vol(F)}{\min\{\glct(F_0), 1\}}.\]
By Lemma~\ref{lem:volume}, 
\begin{align}
 \Vol\left(W, \pi^*K_X-\frac{k}{w}F\right){}&\geq \Vol(W, \pi^*K_X)-\frac{3k}{w} \Vol(F, \pi^*K_X|_F) \nonumber \\
 {}&\geq \Vol(W, K_W)-\frac{3k}{w} \Vol(F, K_F)>0. \label{eq K-k/wF}
 \end{align}

So by Proposition~\ref{prop CCJ fibration}, $W_0$ is minimal as $\frac{k}{w}\geq \frac{1}{w}>\frac{1}{\glct(F_0)}$. 

So after replacing $W$ by $W_0$, we may assume that $W$ is minimal (while $W$ is no longer smooth but with terminal singularities) and a general fiber $F$ of $f$ is a minimal surface of general type.
Recall that $\Vol(W, wK_{W}-kF)>0$ by \eqref{eq K-k/wF}.
So there exists an effective $\mathbb{Q}$-divisor \[G\equiv wK_{W}-kF \equiv wK_{W}-\sum_{i=1}^kF_{i}.\] 
As $F_{i}$ are general, we may assume that $G$ does not contain any $F_i(1\leq i\leq k)$. 
Consider the pair $(W, G+\sum_{i=1}^kF_{i})$. 
By the adjunction formula, we have \[\left. \left(K_X+G+\sum_{i=1}^kF_{i}\right)\right|_{F_{i}}=K_{F_{i}}+G|_{F_{i}}.\]
Note that $G|_{F_{i}}\equiv wK_{F_{i}}$ with $w<{\glct(F)}$, so $({F_{i}}, G|_{F_{i}})$ is klt by the definition of glct (see Remark~\ref{rem glct}). So by the inversion of adjunction \cite[Theorem~5.50]{K-M}, $(W, G+\sum_{i=1}^kF_{i})$ is plt in a neighborhood of $F_{i}$, which implies that $F_i$ is an isolated non-klt center. 
As $W$ has only terminal singularities and $F_{i}$ are general, we can take an open subset $U$ in the smooth locus of $W$ containing $\bigsqcup_{i=1}^k F_{i}$ such that $(W, G+\sum_{i=1}^kF_{i})$ is plt on $U$ with centers $F_{i}$.

Note that 
\[
mK_{W}-K_{W}-G-\sum_{i=1}^kF_{i}\equiv (m-1-w)K_{W}
\]
is nef and big as $w<1$ and $m\geq 2$. 
Take $\phi:W'\to W$ to be a log resolution of $(W, G+\sum_{i=1}^kF_{i})$.
Then by Lemma~\ref{lem nadel2}, 
$H^1(W', \mathcal{F})=0$
where \[\mathcal{F}=\phi_*\OO_{W'}\left(K_{W'}+\left\lceil\phi^*\left(mK_{W}-K_{W}-G-\sum_{i=1}^kF_{i}\right)\right\rceil\right).\]
Note that $\mathcal{F}$ is a subsheaf of $\mathcal{O}_{W}(mK_{W})$, so we have a natural short exact sequence 
\[
0\to \mathcal{F}\to \mathcal{O}_{W}(mK_{W}) \to \mathcal{Q}\to 0,
\]
and the restriction map
\begin{align}
H^0(W, mK_{W}) \to H^0(W, \mathcal{Q})\label{eq W to Q}
\end{align}
is surjective.
Then by Lemma~\ref{lem nadel2}, Lemma~\ref{lem J=I}, and the construction of $U$, we have 
\[
\mathcal{F}|_U\simeq \mathcal{O}_{U}(mK_{W})\otimes \mathcal{J}\left. \left(W, G+\sum_{i=1}^kF_{i}\right)\right|_U\simeq \mathcal{O}_{U}(mK_{W})\otimes \mathcal{I}_{\bigsqcup_{i=1}^kF_{i}}.
\]
This implies that \[\mathcal{Q}|_U\simeq 
\bigoplus_{i=1}^k\mathcal{O}_{U}(mK_{W})|_{F_{i}} \simeq \bigoplus_{i=1}^k\mathcal{O}_{F_{i}}(mK_{F_{i}}),\]
and it is a direct summand of $\mathcal{Q}$ by the support reason. 
So we conclude the proof by \eqref{eq W to Q} as 
$\bigoplus_{i=1}^k H^0(F_{i}, mK_{F_{i}})$ is a direct summand of $H^0(W, \mathcal{Q})$.
\end{proof}

\subsection{An application to fibrations of $(1,2)$-surfaces and $(2,3)$-surfaces} 

As an application of Theorem~\ref{thm:ext for surface}, we prove Theorem~\ref{thm1.3} on fibrations of $(1,2)$-surfaces and $(2,3)$-surfaces. 
First we recall properties of pluricanonical systems of $(1,2)$-surfaces and $(2,3)$-surfaces.
\begin{theorem}\label{surface} Let $S$ be a smooth projective surface of general type. 
\begin{itemize}
\item If $S$ is a $(1,2)$-surface, then $\Phi_{|mK_S|}$ is a generically finite map of degree $2$ for $2\leq m\leq 4$ and is a birational map for $m\geq 5$.
\item If $S$ is a $(2,3)$-surface, then $\Phi_{|mK_S|}$ is a generically finite map of degree $2$ for $1\leq m\leq 3$ and is a birational map for $m\geq 4$.
\end{itemize}
\end{theorem}
\begin{proof}
 This follows from \cite[Theorem~VII.5.1, Proposition~VII.7.1, Proposition~VII.7.2]{BHPV}, \cite[Lemma~1.1]{Horikawa1}, and \cite[Lemma~2.1]{Horikawa2}.
\end{proof}

\begin{proof}[Proof of Theorem~\ref{thm1.3}]

 Denote by $F_0$ the minimal model of $F$. Note that
 \[
 \frac{6\Vol(F)}{\min\{ \glct(F_0),1 \}}\leq 
\begin{cases}
 60 &\text{if } F \text{ is a } (1,2)\text{-surface};\\
 456 &\text{if } F \text{ is a } (2,3)\text{-surface} 
\end{cases}
 \]
by \cite[Theorem~A.1]{Noether} and Lemma~\ref{lemma 2,3}.
Then by Theorem~\ref{thm:ext for surface} with $k=2$, the restriction map
\[
H^0(X, mK_X)\to H^0(F_1,mK_{F_1}) \oplus H^0(F_2,mK_{F_2})
\]
is surjective for any integer $m\geq 2$, where $F_1, F_2$ are distinct general fibers of $f$.
Then the conclusion follows from Proposition~\ref{prop tankeev} and Theorem~\ref{surface}.
\end{proof}

\subsection{Bicanonical maps of fibrations over curves}

In this subsection, we apply Theorem~\ref{thm:ext for surface} to study bicanonical maps of $3$-folds admitting a fibration over a curve. Note that Theorem~\ref{thm:ext for surface} works well if $\Vol(F)$ is relatively small, so we need the following result when $\Vol(F)$ is large.

\begin{proposition}\label{prop:whenvolgeq12}
Let $X$ be a minimal $3$-fold of general type.
Assume that there exists a resolution $\pi:W\to X$ and a fibration $f: W\to C$ to a curve with a general fiber denoted by $F$. 
Assume that $\pi^*K_X-bF$
 is $\mathbb{Q}$-effective for some rational number $b>2$. 
If \[\frac{(b-2)^2}{(b+1)^2} \Vol(F)>8,\] then $d_2(X)\geq 2$.
\end{proposition}

\begin{proof}
We may write $\pi^*K_X-bF\sim_\bQ D\geq 0$ for some effective $\mathbb{Q}$-divisor $D$.
After replacing $W$ by a higher model, we may assume that $F+D$ has simple normal crossing support.
Take $F_1$ and $F_2$ to be $2$ general fibers of $f$. 
Since \[\pi^*K_X-F_1-F_2-\frac{2}{b}D\equiv \left(1-\frac{2}{b}\right)\pi^*K_X\]
is nef and big, 
by the Kawamata--Viehweg vanishing theorem \cite[Theorem~1.2.5]{KMM}, 
\[H^1\left(W, K_W+\left\lceil \pi^*K_X-F_1-F_2-\frac{2}{b}D\right\rceil\right)=0.\]
Then the restriction map 
\begin{align}
& H^0\left(W, K_W+\left\lceil{\pi^*K_X-\frac{2}{b}D}\right\rceil\right) \nonumber \\
\to &H^0(F_1, K_{F_1}+\roundup{L_1})\oplus H^0(F_2, K_{F_2}+\roundup{L_2}) \label{surjFi}
\end{align}
is surjective, where $L_i=(\pi^*K_X-\frac{2}{b}D)|_{F_i}$ for $i=1,2$. 

As $D\geq 0$ and $K_W \geq \pi^*K_X$, we have
\begin{align*}
 H^0\left(W, K_W+\left\lceil \pi^*K_X-\frac{2}{b}D \right\rceil\right) 
 \subset 
 H^0(W, 2K_W).
\end{align*}
So by the surjectivity of
\eqref{surjFi} and Lemma~\ref{lem separete Z1Z2}, it suffices to show that $h^0(F_i, K_{F_i}+\roundup{L_i})\geq 2$ for $i=1,2$, or in other words, $h^0(F, K_{F}+\roundup{L})\geq 2$ for a general fiber $F$ of $f$ where 
\[L=\left. \left(\pi^*K_X-\frac{2}{b}D\right)\right|_{F}\equiv \left(1-\frac{2}{b}\right)\pi^*K_X|_F.\]

By \cite[Corollary~2.3]{Noether}, 
$(1+\frac{1}{b})\pi^*K_X|_F-\sigma^*K_{F_0}$ is $\mathbb{Q}$-effective where $\sigma:F\to F_0$ is the morphism to the minimal model of $F$. Hence by assumption, 
 \begin{align*}
 L^2=\left(1-\frac{2}{b}\right)^2(\pi^*K_X|_F)^2
 \geq \frac{(1-\frac{2}{b})^2}{(1+\frac{1}{b})^2} K_{F_0}^2 >8.
 \end{align*}
 Then by Lemma~\ref{CJY4.3}, we have $h^0(F, K_{F}+\roundup{L})\geq 2$. This completes the proof. 
\end{proof}

The following lemma is derived from \cite[Lemma~4.3]{CJY24}.
\begin{lemma}[{\cite[Lemma~4.3]{CJY24}}]\label{CJY4.3}
Let $F$ be a smooth projective surface of general type
and $L$ be a nef and big $\mathbb{Q}$-divisor on $F$.
If $L^2>8$, then $h^0(F, K_{F}+\roundup{L})\geq 2$.
\end{lemma}
\begin{proof}
If $(L\cdot C)\geq 4$ for any irreducible curve $C$ passing through a very general closed point $P$, then $|K_F+\roundup{L}|$ defines a birational map by \cite[Proposition~4]{Mas99} or \cite[Lemma~2.5]{Chen14}.
In particular,
$h^0(F, K_{F}+\roundup{L})\geq 2.$

Now suppose that there exists an irreducible curve $C$ passing through a very general closed point $P$ such that $(L\cdot C)< 4<\frac{1}{2}L^2$. Since $P$ is very general, $C$ is a nef curve with $p_g(C)\geq 2$. Moreover, after taking a higher model of $F$ and replacing $L$ by its pullback, we may assume that $C$ is smooth; this step does not change $L^2$ or $(L\cdot C)$ by the projection formula. As $L^2>2(L\cdot C)$, there exists a sufficiently small positive number $\epsilon$ such that  $L^2>2(L\cdot (1+\epsilon)C)$, and hence $L-(1+\epsilon)C$ is big by \cite[Theorem~2.2.15]{Positivity1}. In particular, we may write 
$$L\sim_{\mathbb{Q}}(1+\epsilon)C+T,$$ 
where $T$ is an effective $\mathbb{Q}$-divisor.
Then by Sakai's lemma \cite[Example~9.4.12]{Positivity2},
\[
H^1\left(F, K_F+\left\lceil{L-C-\frac{1}{1+\epsilon}T}\right\rceil\right)=0,
\]
which implies that the natural map
\[
H^0\left(F, K_F+\left\lceil{L-\frac{1}{1+\epsilon}T}\right\rceil\right)\to H^0\left(C, K_C+\left. \left\lceil{L-\frac{1}{1+\epsilon}T-C}\right\rceil \right|_{C}\right)
\]
is surjective.
Since $\deg(\roundup{L-\frac{1}{1+\epsilon}T-C}|_C)>0$, it is clear that
\[h^0\left(C, K_C+\left. \left\lceil{L-\frac{1}{1+\epsilon}T-C}\right\rceil \right|_C\right)\geq 2\]
by the Riemann--Roch formula. Therefore, 
\[
h^0(F, K_F+\roundup{L})\geq h^0\left(F, K_F+\left\lceil{L-\frac{1}{1+\epsilon}T}\right\rceil\right)\geq 2. \qedhere\]
\end{proof}
Now we are ready to study the bicanonical maps. 
\begin{corollary}\label{cor X/C}
Let $X$ be a minimal $3$-fold of general type.
Assume that there exists a resolution $\pi:W\to X$ and a fibration $f:W\to C$ to a curve with a general fiber denoted by $F$. 
Assume that $ \Vol(X)>2.9\times10^4$ and $\pi^*K_X-bF$ is $\mathbb{Q}$-effective
for some rational number $b> 24$, then $d_2(X)\geq 2$.
\end{corollary}

\begin{proof}
 If $\Vol(F)\geq 12$, then 
 \[
 \frac{(b-2)^2}{(b+1)^2} \Vol(F)> \frac{22^2}{25^2}\cdot 12>8
 \]
 as $b>24$.
 Then the conclusion follows from Proposition \ref{prop:whenvolgeq12}.
 
 If $\Vol(F)\leq 11$, then
 $\glct(F_0)\geq \frac{1}{429}$ by Lemma \ref{429} where $F_0$ is the minimal model of $F$.
Hence by assumption, 
 \[
 \Vol(W)=\Vol(X)>2.9\times 10^4>\frac{6\Vol(F)}{\min \{ \glct(F_0),1 \}}.
 \]
 Then by Theorem \ref{thm:ext for surface} with $k=2$ and $m=2$, the restriction map
 \[
 H^0(W,2K_W) \to H^0(F_1,2K_{F_1}) \oplus H^0(F_2,2K_{F_2})
 \]
 is surjective, where $F_1,F_2$ are distinct general fibers of $f$.
 Moreover, $h^0(F, 2K_F)\geq 2$ for a general fiber $F$ of $f$ by the Riemann--Roch formula (see \cite[Page~185]{Bombieri}).
 Hence the conclusion follows from Lemma \ref{lem separete Z1Z2}. Here note that $\Phi_{|2K_W|}=\Phi_{|2K_X|}\circ \pi$.
 \end{proof}

\section{An extension theorem for non-klt centers of dimension $1$}\label{subsection4.2}
In this section, we prove an extension theorem for non-klt centers of dimension $1$ (Theorem~\ref{thmcurv}).

\subsection{Inversion of adjunction on curves}
We recall the following special case of inversion of adjunction for non-klt centers of dimension $1$. For the general version of inversion of adjunction, we refer the readers to \cite{FH23}.
\begin{proposition}\label{prop adjandinvadj}
 Let $(X,\Delta)$ be a projective pair and let $C$ be a curve on $X$ such that $(X, \Delta)$ is lc near the generic point of $C$ and $C$ is a non-klt center of $(X, \Delta)$.
 Then \[
(K_X+\Delta)\cdot C\geq 2p_g(C)-2.
\]
Furthermore, if there is a unique non-klt place of $(X,\Delta)$ over $C$ and
\[
(K_X+\Delta)\cdot C<2p_g(C)-1,
\]
then the following assertions hold:
\begin{enumerate}
 \item $C$ is smooth;
 \item $(X, \Delta)$ is lc in a neighborhood of $C$ and $C$ is an isolated non-klt center of $(X, \Delta)$; 
 \item $\mathcal{J}(X,\Delta)=\mathcal{I}_{C\sqcup Z}$ for some closed subscheme $Z$ disjoint from $C$;

 \item if $D$ is a Cartier divisor on $X$ such that $D-K_X-\Delta$ is nef and big, then the natural restriction map 
 \[H^0(X, \mathcal{O}_X(D)\otimes \mathcal{I}_Z)\to H^0(C, D|_C)\]
is surjective. 
 
\end{enumerate}

\end{proposition}
\begin{proof}
 Let $\nu: C^\nu \to C$ be the normalization of $C$.
 Since $(X, \Delta)$ is lc near the generic point of $C$ and $C$ is a non-klt center of $(X,\Delta)$, 
by \cite[Theorem~1.1, Theorem~1.2]{FH23}, there is an adjunction formula
\[
\nu^*(K_X+\Delta)=K_{C^\nu}+B_{C^\nu}+M_{C^\nu},
\]
where $B_{C^\nu}\geq 0$ and $M_{C^\nu}$ is semi-ample. Hence 
\[
(K_X+\Delta)\cdot C\geq \deg K_{C^\nu}=2p_g(C)-2.
\]

From now on, suppose that there is a unique non-klt place of $(X,\Delta)$ over $C$ and
\[
(K_X+\Delta)\cdot C<2p_g(C)-1.
\]
Then
 \begin{align*} \deg(B_{C^\nu})&=\deg(\nu^*(K_X+\Delta))-\deg(K_{C^\nu})-\deg(M_{C^\nu}) \\
 &\leq \deg(\nu^*(K_X+\Delta))-\deg(K_{C^\nu}) \\
 &=(K_X+\Delta)\cdot C-2p_g(C)+2<1. 
 \end{align*}
In particular, all coefficients of $B_{C^\nu}$ are smaller than $1$, and hence $({C^\nu}, B_{C^\nu}+M_{C^\nu})$ is an {\it NQC generalized klt
pair} (see \cite[Theorem~1.2]{FH23} for the definition). We only need to recall that as $C^\nu$ is a smooth curve, $({C^\nu}, B_{C^\nu}+M_{C^\nu})$ is an NQC generalized klt
pair if and only if all coefficients of $B_{C^\nu}$ are smaller than $1$.

Then by \cite[Theorem~1.1]{FH23}, $(X, \Delta)$ is lc in a neighborhood of $C$ and $C$ is a minimal non-klt center
of $(X,\Delta)$. In particular, $C$ is normal by Proposition~\ref{prop non-klt center}(2). 
As there is a unique non-klt place over $C$, $C$ is a maximal non-klt center by Proposition~\ref{prop non-klt center}(4). Hence it is an isolated non-klt center by Proposition~\ref{prop non-klt center}(3).

By Lemma~\ref{lem J=I}, we have $\mathcal{J}(X, \Delta)=\mathcal{I}_C$ in a neighborhood of $C$, so $\mathcal{J}(X,\Delta)=\mathcal{I}_{C\sqcup Z}$ for some closed subscheme $Z$ disjoint from $C$.
Then there is a short exact sequence 
\[
0\to \mathcal{J}(X,\Delta)\to \mathcal{I}_Z\to \mathcal{O}_C\to 0.
\]
By the Nadel vanishing theorem (Theorem~\ref{nadel}), 
the natural restriction map 
 \[H^0(X, \mathcal{O}_X(D)\otimes \mathcal{I}_Z)\to H^0(C, D|_C)\]
is surjective. 
\end{proof} 

\subsection{An extension theorem for non-klt centers of dimension $1$}\label{subsection5.2}
In this subsection, we apply Proposition~\ref{prop adjandinvadj} to study bicanonical maps of $3$-folds covered by non-klt centers of dimension $1$ of genus $2$.
\begin{theorem}\label{thmcurv}
Let $X$ be a minimal $3$-fold of general type and let $\pi:W\to X$ be a resolution.
Suppose that there are curves $C_1, C_2$ on $W$ and effective $\mathbb{Q}$-divisors $\Delta_1, \Delta_2 $ on $X$ with the following properties for $i=1,2$:
\begin{itemize} 
 
 \item $C_i$ is not contained in the exceptional locus of $\pi$;
 \item $(X, \Delta_i)$ is lc near the generic point of $\pi(C_i)$, and $\pi(C_i)$ is a non-klt center of $(X, \Delta_i)$ with a unique non-klt place over $\pi(C_i)$;

 \item $p_g(C_i)=2$ and $K_W\cdot C_i\leq 2$;
 \item $\Delta_i\sim_\mathbb{Q}\delta_iK_X$ for some positive rational number $\delta_i$ with $\delta_1+\delta_2 \leq \frac{1}{3};$

 \item $\pi(C_2)\not\subset \Supp(\Delta_1)$.
\end{itemize} 
Then the following assertions hold:
\begin{enumerate}
 \item $C_1, C_2$ are smooth.
 \item The natural restriction map \begin{align}
 H^0(W, 2K_W) \to H^0(C_1,2K_{W}|_{C_1})\oplus H^0(C_2,2K_{W}|_{C_2}) \label{surj C1C2}
 \end{align}
 is surjective. 
 \item 
 $d_2(X)=d_2(W)\geq 2$. 
\end{enumerate} 
\end{theorem}

\begin{proof} 
We may write $K_W=\pi^*K_X+E_\pi$ where $E_\pi$ is an effective $\pi$-exceptional $\mathbb{Q}$-divisor.

For $i=1,2$, by applying Proposition~\ref{prop adjandinvadj} to $(X, \Delta_i)$ and $\pi(C_i)$, we have 
\begin{align*}
 (1+\delta_i)K_X\cdot \pi(C_i)=(K_X+\Delta_i)\cdot \pi(C_i)\geq 2p_g(C_i)-2=2.
\end{align*}
Then by $\delta_1+\delta_2 \leq \frac{1}{3}$, we have
\begin{align*}
 &(K_W+\pi^*(\Delta_1+\Delta_2)+E_\pi)\cdot C_i=2K_W\cdot C_i-(1-\delta_1-\delta_2)K_X\cdot \pi(C_i)\\
 \leq &4- \frac{2(1-\delta_1-\delta_2)}{1+\delta_i}<4- \frac{2(1-\delta_1-\delta_2)}{1+\delta_1+\delta_2}\leq 3=2p_g(C_i)-1.
\end{align*} 
In particular, 
\begin{align}
 & (K_W+\pi^*\Delta_1+E_\pi)\cdot C_1 <2p_g(C_1)-1,\label{eq KC<3 C1}\\
 & (K_W+\pi^*(\Delta_1+\Delta_2)+E_\pi)\cdot C_2 <2p_g(C_2)-1.\label{eq KC<3 C2}
\end{align}
By assumption, $C_1$ and $C_2$ are not contained in $\Supp(E_\pi)$ and $C_2\not\subset \Supp(\pi^*\Delta_1)$. So we know that
\begin{itemize}
 \item $(W, \pi^*\Delta_1+E_\pi)$ is lc near the generic point of $C_1$, and $C_1$ is a non-klt center of $(W, \pi^*\Delta_1+E_\pi)$ with a unique non-klt place over $C_1$;
 \item $(W, \pi^*(\Delta_1+\Delta_2)+E_\pi)$ is lc near the generic point of $C_2$, and $C_2$ is a non-klt center of $(W, \pi^*(\Delta_1+\Delta_2)+E_\pi)$ with a unique non-klt place over $C_2$. 
\end{itemize} 
Also 
\begin{align*}
 & 2K_W-K_W-\pi^*\Delta_1-E_\pi\sim_\bQ(1-\delta_1)\pi^*K_X,\\
 & 2K_W-K_W-\pi^*(\Delta_1+\Delta_2)-E_\pi\sim_\bQ(1-\delta_1-\delta_2)\pi^*K_X
\end{align*}
are nef and big. 

Applying Proposition~\ref{prop adjandinvadj} with \eqref{eq KC<3 C1} to $(W, \pi^*\Delta_1+E_\pi)$ and $C_1$, 
we get that $$\mathcal{J}(W, \pi^*\Delta_1+E_\pi)=\mathcal{I}_{C_1\sqcup Z_1}$$ for some closed subscheme $Z_1$ disjoint from $C_1$ and the natural map
\[
H^0(W, \mathcal{O}_W(2K_W)\otimes \mathcal{I}_{Z_1})\to H^0(C_1, 2K_W|_{C_1})
\]
is surjective. 
This implies that 
 \begin{align}
 H^0(W, 2K_W)\to H^0(C_1, 2K_W|_{C_1}) \label{eq surj C1}
\end{align}
is surjective.

Applying Proposition~\ref{prop adjandinvadj} with \eqref{eq KC<3 C2} to $(W, \pi^*(\Delta_1+\Delta_2)+E_\pi)$ and $C_2$, 
we get that $$\mathcal{J}(W, \pi^*(\Delta_1+\Delta_2)+E_\pi)=\mathcal{I}_{C_2\sqcup Z_2}$$ for some closed subscheme $Z_2$ disjoint from $C_2$ and the natural map 
\[
H^0(W, \mathcal{O}_W(2K_W)\otimes \mathcal{I}_{Z_2})\to H^0(C_2, 2K_W|_{C_2})
\]
is surjective. 
By construction, $(W, \pi^*(\Delta_1+\Delta_2)+E_\pi)$ is not klt at $C_1$ as $C_1$ is a non-klt center of $(W, \pi^*\Delta_1+E_\pi)$, hence 
 $C_1\subset \Supp(Z_2)$ and $\mathcal{I}_{Z_2}\subset \mathcal{I}_{C_1}$.
This implies that 
\begin{align}
 H^0(W, \mathcal{O}_W(2K_W)\otimes \mathcal{I}_{C_1})\to H^0(C_2, 2K_W|_{C_2})\label{eq surj C2} 
\end{align}
is surjective.

Combining the surjectivity of \eqref{eq surj C1} and \eqref{eq surj C2}, we get the surjectivity of \eqref{surj C1C2} by Lemma~\ref{lem rest 5 lemma}. 
Since $\deg K_W|_{C_i}=2$, 
we have 
$h^0(C_i, 2K_W|_{C_i})=3$ for $i=1,2$ by the Riemann--Roch formula. 
Hence $d_2(W)\geq 2$ by Lemma~\ref{lem separete Z1Z2}.
\end{proof}

We need the following well-known result on the degree of covering family of curves. 
\begin{lemma}\label{lem degKW}
Let $W, V, T$ be smooth projective varieties with a fibration  $f: V\to T$  of relative dimension $1$ and a surjective morphism $\phi: V\to W$. Suppose that the induced map $F\to W$ is birational onto its image for a general fiber $F$ of $f$. Then \[K_W\cdot \phi(F)\leq 2p_g(F)-2.\]
\end{lemma}
\begin{proof}
By \cite[Lemma~2.28]{Bir19}, we may assume that $\phi$ is generically finite after cutting $T$ by hyperplane sections. Then by the Hurwitz formula, 
$K_V-\phi^*K_W$ is effective. As $F$ is general, by the projection formula we get $K_W\cdot \phi(F)\leq K_V\cdot F =2p_g(F)-2.$ 
\end{proof}

\section{Proof of Theorem~\ref{thm1.1}}\label{section7}

In this section, we prove Theorem~\ref{thm1.1}.

Firstly we make preparation on constructing a family of non-klt centers. The following lemma allows us to cut down the dimension of a non-klt center of codimension $1$ (See \cite{CJ17}). 

\begin{lemma}\label{genefam}
Let $X$ be a normal projective variety of dimension $d$ and let $A$ be an ample $\bQ$-divisor.
Let $Z$ be a $\bQ$-Cartier $\bQ$-divisor on $X$ such that for any general closed point $x\in X$, there are distinct prime divisors $Z_x$ and $Z'_x$ containing $x$ such that $Z_x\equiv Z'_x\equiv Z$. 
Fix a rational number $0<a\ll 1$.
Then for general closed points $x, y\in X$, 
after possibly switching $x,y$, there is an effective $\mathbb{Q}$-divisor $\Delta\sim_\bQ b Z+aA$ with $b\leq 4$ such that 
\begin{itemize}
 \item $(X,\Delta)$ is not klt near $y$ but it is lc near $x$ with a unique non-klt place whose center contains $x$, and the center is denoted by $G$;
 \item $\dim G<d-1$; 
 \item $(aA|_G)^{\dim G}\leq (4d^2)^d$ if $\dim G>0$.
\end{itemize} 
\end{lemma}

\begin{proof}
We may assume that $x, y$ are smooth points of $X$.
Take $Z_0$ to be the reduced part of $Z_x+Z'_x+Z_y+Z'_y$. 
Then $Z_0\equiv b_0Z$ with $2\leq b_0\leq 4$.
Set 
\[
c_0:=\min\{c\mid (X, cZ_0) \text{ is not klt near } x \text{ and } y\}.
\]
After possibly switching $x,y$, we may assume that $(X, c_0Z_0)$ is lc near $x$. 
Take $G_x$ to be the minimal non-klt center of $(X, c_0Z_0)$ containing $x$, then $\dim G_x<\dim X-1$: indeed, if $c_0<1$, then this is clear; if $c_0=1$, then this follows from Proposition~\ref{prop non-klt center}(1).

We claim that either $y\in G_x$ or $(X, c_0Z_0)$ has a non-klt center containing $y$ but not $x$. 
If $(X, c_0Z_0)$ is not lc near $y$, then $\{y\}$ is a non-klt center of $(X, c_0Z_0)$ by the proof of \cite[Corollary~2.31(1)]{K-M}.
If $(X, c_0Z_0)$ is lc near $y$, taking $G_y$ to be the minimal non-klt center of $(X, c_0Z_0)$ containing $y$, then $\dim G_y<\dim X-1$ by the same argument as in the first paragraph. 
After possibly switching $x, y$, we may assume that $\dim G_y\leq \dim G_x$. Then by the minimality of $G_x$, we know that either $G_y=G_x$ or $x\not \in G_y$. So the claim is proved. 

Then by \cite[Lemma~2.16]{Bir23}, there exist rational numbers $0 \leq t \ll s \leq 1$ and an effective $\mathbb{Q}$-divisor $\Delta'\sim_\mathbb{Q}sc_0Z_0+tA$ such that $(X,\Delta')$ is not klt near $y$ but it is lc near $x$ with a unique non-klt place whose center contains $x$, and the
center of this non-klt place is $G:=G_x$.

If $\dim G>0$ and $(\frac{a}{4d}A|_G)^{\dim G} > d^d$, then by \cite[Lemma~2.18]{Bir23}, we may replace $\Delta'$ by $\Delta''\sim_{\bQ} sc_0Z_0+tA+\frac{a}{2d}A$ and replace $G$ by a proper subvariety $G'\subsetneq G$. This process can be repeated only $d-2$ times, so in the end we may assume that 
$\Delta'\sim_\mathbb{Q}sc_0Z_0+t'A$ for some $t\leq t'< t+\frac{a}{2}$ and 
$(\frac{a}{4d}A|_G)^{\dim G} \leq d^d$ if $\dim G>0$.

Here note that $\Delta'\sim_\mathbb{Q}sc_0Z_0+t'A\equiv sc_0b_0 Z+t'A$ with $sc_0b_0\leq 4$ and $t'<a$ as $t\ll 1$. 
Then we conclude the proof by taking $\Delta:=\Delta'+A'\sim_\mathbb{Q}sc_0b_0 Z+aA$ where $A'$ is a general effective ample $\mathbb{Q}$-divisor with $A'\sim_{\mathbb{Q}} (a-t')A+(sc_0b_0Z+tA-\Delta')$ whose support does not contain $x,y$. 
\end{proof}

\begin{definition}[{See \cite[\S\,2.14]{Bir23}}]
 Let $X$ be a normal projective variety. A {\it bounded
family} $\mathcal{G}$ {\it of subvarieties} of $X$ is a family of closed subvarieties such that 
there is a projective morphism $V\to T$ between schemes of finite type and a morphism $V\to X$ which embeds the fibers of $V\to T$ over closed points into $X$, and each member of
the family $\mathcal{G}$ is isomorphic to a fiber of $V\to T$ over some closed point. Such morphism can be constructed from the Hilbert schemes or the Chow varieties of subvarieties with bounded degree, see \cite[Chapter~1]{KollarRC}. 
\end{definition}

Now we construct a bounded family of non-klt centers on $3$-folds with large canonical volumes.
\begin{proposition}\label{prop G and Delta}
 Let $X$ be a minimal $3$-fold of general type.
 Suppose that $(\delta K_X)^3>6^3$ for some positive rational number $\delta>0$.
Then there is a bounded family $\mathcal{G}$
of subvarieties of $X$ and a positive rational number $\delta'>0$ such that 
\begin{itemize}
 \item for each pair $x,y\in X$ of general closed points, after possibly switching $x,y$, there is a member $G_{x, y}$ of the family $\mathcal{G}$ and an effective $\mathbb{Q}$-divisor $\Delta_{x, y}\sim_\bQ \delta'K_X$ such that $(X,\Delta_{x, y})$ is not klt near $y$ but it is lc near $x$ with a unique non-klt place whose center contains $x$, and this
center is $G_{x, y}$;

\item one of the following holds:
\begin{enumerate}
 \item $\delta'< \delta$ and there exists a resolution $\pi:W\to X$ and a fibration $f:W\to C$ to a curve such that $\delta'\pi^*K_X-F$ is big, where $F$ is a general fiber of $f$;

 \item $\delta'< 4\delta$, $\dim G_{x, y}=1$, and $p_g(G_{x, y})=2$;

 \item $\delta'< 6\delta+\frac{1}{2}$ and $\dim G_{x, y}=0$.
\end{enumerate}
\end{itemize}
\end{proposition}

\begin{proof}
Fix a sufficiently small rational number $\epsilon>0$ such that $((\delta-\epsilon) K_X)^3>6^3$.
By \cite[Proposition~2.61]{K-M}, there exists an effective $\bQ$-divisor $E$ such that $A_k:=K_X-\frac{1}{k}E$ is ample for any sufficiently large integer $k$.

 By \cite[\S\,2.15(2)]{Bir23}, 
 there is a bounded family $\mathcal{G}$
of subvarieties of $X$ such that 
for each pair $x,y\in X$ of general closed points, there is a member $G_{x, y}$ of the family $\mathcal{G}$ and an effective $\mathbb{Q}$-divisor $ \Delta_{x, y}\sim_\bQ (\delta-\epsilon) K_X$ such that $(X,\Delta_{x, y})$ is not klt near $y$ but it is lc near $x$ with a unique non-klt place whose center contains $x$, and this
center is $G_{x, y}$. 

Set $\delta': =\delta-\epsilon$. In the following we will discuss on $\dim G_{x, y}$ and will modify $\mathcal{G}, \Delta_{x, y}, \delta'$ to satisfy our requirement.

\medskip

{\bf Case 1}. $\dim G_{x, y}=2$ and $\delta'<\delta$.

In this case, $\delta' K_X-G_{x, y}\sim_{\bQ}\Delta_{x, y} -G_{x, y}\geq 0$. After replacing $\delta'$ by $\delta'+\frac{\epsilon}{2}$, we may assume that $\delta' K_X-G_{x, y}$ is big and we fix an effective $\bQ$-divisor $E'\sim_{\bQ} \delta' K_X-G_{x, y}$.
By taking a birational modification of the family $\mathcal{G}$, 
we get a fibration $f:W\to T$ between smooth projective varieties and a surjective morphism $\pi:W\to X$ such that for a general fiber $F$ of $f$, $\pi$ maps $F$ birationally onto its image, $\pi_*F\equiv G_{x, y}$, and $\delta' K_X-\pi_*F$ is big. We may assume that $\phi$ is generically finite after cutting $T$ by hyperplane sections. Here we may assume that $W$ is constructed from 
the Hilbert scheme of subvarieties of $X$, which means that different general fibers of $\pi$ are mapped to  different subvarieties of $X$. 

If $\pi$ is birational, then $\delta'\pi^*K_X-F$ is big as $\delta' K_X-G_{x, y}$ is big. In this case we get (1).

If $\pi$ is not birational, then 
$\pi$ is generically finite of degree at least $2$. Since $W$ is constructed from the Hilbert scheme of subvarieties of $X$,
for any general closed point $x\in X$, there are distinct prime divisors $G_x$ and $G'_x$ containing $x$ such that $G_x\equiv G'_x\equiv G_{x, y}$. 

Then by Lemma~\ref{genefam},
we can replace $(\Delta_{x, y}, G_{x, y}, \delta')$ by a new triple $(\Delta'_{x, y}, G'_{x, y}, b\delta'+a)$ where $\Delta'_{x, y}$ is an effective $\bQ$-divisor satisfying \[\Delta'_{x, y}\sim_\bQ bG_{x, y}+aA_k+\frac{a}{k}E+bE'\sim_\bQ
(b\delta'+a)K_X
\] for some $b\leq 4$ and $0<a\ll 1$, $\dim G'_{x, y}\leq 1$, and $G'_{x, y}$ belongs to a bounded family (of curves or points), which reduces the discussion to Case 2.
Here adding $\frac{a}{k}E+bE'$ does not affect the singularities near $x,y$.

\medskip

{\bf Case 2}. $\dim G_{x, y}\leq 1$ and $\delta'<4\delta$.

If $\dim G_{x, y}=0$, then we get (3). 

If $\dim G_{x, y}=1$ and $p_g(G_{x, y})=2$, then we get (2).

Now suppose that $\dim G_{x, y}=1$ and $p_g(G_{x, y})\geq 3$. Then \[(1+\delta')K_X\cdot G_{x, y}=(K_X+\Delta_{x, y})\cdot G_{x, y}\geq 2p_g(G_{x, y})-2\geq 4\] by Proposition~\ref{prop adjandinvadj}.
This implies that \[\frac{1+4\delta}{4}A_{k}\cdot G_{x, y}>1=(\dim G_{x, y})^{\dim G_{x, y}}\] for sufficiently large $k$.
Then by \cite[Lemma~2.18]{Bir23},
we can replace $(\Delta_{x, y}, G_{x, y}, \delta')$ by a new triple $(\Delta'_{x, y}, G'_{x, y}, \delta'+\frac{1+4\delta}{2})$ where $\Delta'_{x, y}$ is 
an effective $\bQ$-divisor satisfying \[\Delta'_{x, y}\sim_\bQ\Delta_{x, y}+\frac{1+4\delta}{2} A_{k}+\frac{1+4\delta}{2k}E_k\sim_\bQ \left(\delta'+\frac{1+4\delta}{2} \right)K_X\] 
and $\dim G'_{x, y}=0$. Then we get (3).

Here we remark that in \cite[Lemma~2.18]{Bir23}, the assumption ``$\textrm{vol}(A|_G) > d^d$'' (where $d=\dim X$) can be replaced by ``$\textrm{vol}(A|_G) > (\dim G)^{\dim G}$'' as its proof is based on \cite[Theorem~6.8.1]{SOP}.
\end{proof}

\begin{proof}[Proof of Theorem~\ref{thm1.1}]
After replacing $X$ by its minimal model, we may assume that $X$ is a minimal $3$-fold with $\Vol(X)=K_X^3>12^6$.
Then we can
apply Proposition~\ref{prop G and Delta} by taking $\delta=\frac{1}{24}$ and get $\mathcal{G}, \Delta_{x,y}, G_{x, y}, \delta'$ correspondingly.

If Proposition~\ref{prop G and Delta}(1) holds, then by Corollary~\ref{cor X/C} for $b>\delta^{-1}=24$,
$d_2(X)\geq 2$.

If Proposition~\ref{prop G and Delta}(2) holds for general $x,y\in X$, 
then by taking general $x_1, y_1, x_2, y_2\in X$, we can get \[(C_1, C_2, \Delta_1, \Delta_2):=(G_{x_1, y_1}, G_{x_2, y_2}, \Delta_{x_1, y_1}, \Delta_{x_2, y_2})\]
satisfying all assumptions of Theorem~\ref{thmcurv} with $\delta_1=\delta_2=\delta'<4\delta=\frac{1}{6}$. Here $K_W\cdot C_i\leq 2$ is by Lemma~\ref{lem degKW}; $\pi(C_2)\not \subset \Supp(\Delta_1)$ as we may choose general $x_2$ such that $x_2\not \in \Supp(\Delta_1)$. Then by Theorem~\ref{thmcurv}, $d_2(X)\geq 2$.

If Proposition~\ref{prop G and Delta}(3) holds for general $x,y\in X$, then
$K_X$ is potentially birational as $\delta'<6\delta+\frac{1}{2}<1$ (see \cite[\S\,2.12]{Bir23} or \cite[Definition~2.3.3]{HMX13}).
So
$\Phi_{|2K_X|}$ is birational by \cite[Definition~2.3.4]{HMX13} and hence $d_2(X)=3 > 2$ .

In summary, $d_2(X)\geq 2$ holds in each case and the proof is completed.
\end{proof}

\section*{Acknowledgments} 
The authors would like to thank Meng Chen for helpful discussions. 
The second author would like to thank Hexu Liu for helpful discussions and suggestions in subsection \ref{subsection5.2}.

This work was supported by National Key Research and Development Program of China (No. 2023YFA1010600, No. 2020YFA0713200) and NSFC for Innovative Research Groups (No. 12121001). C.~Jiang is a member of the Key Laboratory of Mathematics for Nonlinear Sciences, Fudan University.

\end{document}